\documentclass[12pt, a4 paper]{article}
\usepackage{amssymb}
\usepackage{amsmath}
\usepackage{amsmath,amssymb,amsbsy,amsfonts,amsthm,latexsym,
amsopn,amstext,amsxtra,euscript,amscd}
\usepackage{color}

\newtheorem{lem}{Lemma}
\newtheorem{lemma}[lem]{Lemma}
\newtheorem{prop}{Proposition}
\newtheorem{proposition}[prop]{Proposition }
\newtheorem{thm}{Theorem}
\newtheorem{theorem}[thm]{Theorem}
\newtheorem{cor}{Corollary}
\newtheorem{corollary}[cor]{Corollary}
\newtheorem{prob}{Problem}
\newtheorem{problem}[prob]{Problem}

\newtheorem{con}{Conjecture}
\newtheorem{conjecture}[con]{Conjecture}

\newcommand{\F}{{\mathbb F}}
\newcommand{\Z}{{\mathbb Z}}

\newcommand{\beeq}{\begin{eqnarray*}}
\newcommand{\eneq}{\end{eqnarray*}}

\def\Z{\mathbb Z}

\multlinegap=0em \DeclareFontFamily{T1}{msb}{}
\DeclareFontShape{T1}{msb}{m}{ol}{<5> <6> <7> <8> <9> gen * msbm
<10> <10.95> <12> <14.4> <17.28> <20.74> <24.88> msbm10}{}
\DeclareSymbolFont{AMSb}{T1}{msb}{m}{ol} \multlinegap=0em
\begin{document}

\title{Concentration points on two and three dimensional modular hyperbolas and applications}

\author{
{\sc J. Cilleruelo} \\
{Instituto de Ciencias Matem\'{a}ticas (CSIC-UAM-UC3M-UCM) and} \\{Departamento de Matem\'{a}ticas} \\
{Universidad Aut\'{o}noma de Madrid} \\
{Madrid-28049, Spain} \\
{\tt franciscojavier.cilleruelo@uam.es} \\
\and
{\sc M.~Z.~Garaev} \\
{Instituto de Matem\'{a}ticas}\\
{ Universidad Nacional Aut\'onoma de M\'{e}xico} \\
{C.P. 58089, Morelia, Michoac\'{a}n, M\'{e}xico} \\
{\tt garaev@matmor.unam.mx}}

\date{}

\maketitle

\begin{abstract}
Let $p$ be a large prime number, $K,L,M,\lambda$ be integers with
$1\le M\le p$ and  $\gcd(\lambda,p)=1.$ The aim of our paper is to
obtain sharp upper bound estimates for the number $I_2(M; K,L)$ of
solutions of the congruence
$$
xy\equiv\lambda \pmod p, \qquad K+1\le x\le K+M,\quad L+1\le y\le
L+M
$$
and for the number $I_3(M;L)$ of solutions of the congruence
\begin{equation}
\label{three} xyz\equiv\lambda\pmod p, \quad L+1\le x,y,z\le L+M.
\end{equation}
Using the idea of Heath-Brown from~\cite{HB}, we obtain a bound for
$I_2(M;K,L),$ which improves several recent results of Chan and
Shparlinski~\cite{SC}. For instance, we prove that if $M<p^{1/4},$
then $I_2(M;K,L)\le M^{o(1)}.$

The problem with $I_3(M;L)$ is more difficult and requires a
different approach. Here, we connect this problem with the Pell
diophantine equation and prove that for $M<p^{1/8}$ one has
$I_3(M;L)\le M^{o(1)}.$ Our results have applications to some other
problems as well. For instance, it follows that if $\mathcal{I}_1,
\mathcal{I}_2, \mathcal{I}_3$ are intervals in $\F^*_p$ of length
$|\mathcal{I}_i|< p^{1/8},$ then
$$
|\mathcal{I}_1\cdot \mathcal{I}_2\cdot \mathcal{I}_3|=
(|\mathcal{I}_1|\cdot |\mathcal{I}_2|\cdot
|\mathcal{I}_3|)^{1-o(1)}.
$$

\end{abstract}

\hspace{1cm}{\bf MSC Classification: 11A07, 11B75}

\newpage

\section{Introduction} In what follows, $p$ denotes a large prime number, $K,L,M,\lambda$ are integers with
$1\le M\le p$ and  $\gcd(\lambda,p)=1.$ By $x,y,z$ we denote
variables that take integer values. The notation $B^{o(1)}$ denotes
such a quantity that for any $\varepsilon>0$ there exists
$c=c(\varepsilon)>0$ such that $B^{o(1)}<cB^{\varepsilon}.$

Let $I_2(M; K,L)$ be the number of solutions of the congruence
$$
xy\equiv \lambda\pmod p,\qquad K+1\le x\le K+M,\quad L+1\le y\le L+M
$$
and let $I_3(M; L)$ be the number of solutions of the congruence
$$
xyz\equiv \lambda\pmod p,\qquad L+1\le x,y,z\le L+M{\color{red}.}
$$
 Estimates of incomplete Kloosterman sums implies that
\begin{equation}\label{K}
I_2(M; K,L)=\frac{M^2}{p}+O(p^{1/2}(\log p)^2).
\end{equation}
In particular, if $M/(p^{3/4}(\log p)^2)\to\infty$ as $p\to\infty,$
one gets that
$$
I_{2}(M; K,L)=(1+o(1))\frac{M^2}{p}.
$$
This asymptotic formula also holds  when $M/p^{3/4}\to\infty$ as
$p\to\infty$ (see \cite{G}). The problem of upper bound estimates of
$I_{2}(M; K,L)$ for smaller values of $M$ has been a subject of the
work of Chan and Shparlinski \cite{SC}. Using Bourgain's sum-product
estimate \cite{B}, they have shown that there exists an effectively
computable constant $\eta > 0$ such that for any positive integer $M
< p$, uniformly over arbitrary integers $K$ and $L$, the following
bound holds:
$$
I_2(M; K,L)\ll \frac{M^2}{p}+M^{1-\eta}.
$$

In the present paper we obtain the following upper bound estimates
for $I_2(M; K,L).$

\begin{theorem}
\label{T1} Uniformly over arbitrary integers $K$ and $L$, we have
\begin{equation}
I_2(M; K,L)< \frac{M^{4/3+o(1)}}{p^{1/3}}+M^{o(1)}.\label{N}
\end{equation}
When $K=L,$ we have
\begin{equation}
I_2(M; L, L)< \frac{M^{3/2+o(1)}}{p^{1/2}}+M^{o(1)}.\label{Ne}
\end{equation}
\end{theorem}

In particular, if $M<p^{1/4}$ then $I_2(M; K,L) < M^{o(1)}.$

Theorem~\ref{T1} together with \eqref{K} easily implies the
following consequence, which improves upon the mentioned result of
Chan and Shparlinski.

\begin{corollary}
\label{T2} Uniformly over arbitrary integers $K$ and $L$, we have
$$
I_2(M; K,L)\ll \frac{M^2}{p}+M^{4/5+o(1)}.
$$
If $K=L,$ then
$$
I_2(M; L,L)\ll \frac{M^2}{p}+M^{3/4+o(1)}.
$$
\end{corollary}

The proof of Theorem~\ref{T1} is based on an idea of
Heath-Brown~\cite{HB}. The problem with $I_3(M;L)$ is more difficult
and requires a different approach. Here, we shall connect this
problem with the Pell diophantine equation and establish the
following statement.

\begin{theorem}\label{T4}
Let $M \ll p^{1/8}.$ Then, uniformly over arbitrary integer $L,$ we
have
\begin{equation}
I_3(M; L)\ll M^{o(1)}.
\end{equation}
\end{theorem}

\smallskip

From Theorem~\ref{T4} we can easily derive a sharp bound for the
cardinality of product of three small intervals in $\F^*_p$.

\begin{corollary}\label{T7}
Let $\mathcal{I}_1, \mathcal{I}_2, \mathcal{I}_3$ be intervals in
$\F^*_p$ of length $|\mathcal{I}_i|< p^{1/8}.$ Then
$$
|\mathcal{I}_1\cdot \mathcal{I}_2\cdot \mathcal{I}_3|=
(|\mathcal{I}_1|\cdot |\mathcal{I}_2|\cdot
|\mathcal{I}_3|)^{1-o(1)}.
$$
\end{corollary}

Theorems~\ref{T1} and~\ref{T4} have also applications to the problem
on concentration points on exponential curves as well. Let $g\ge 2$
be an integer of multiplicative order $t,$ and let $M<t.$ Denote by
$J_a(M; K,L)$ the number of solutions of the congruence
$$
y\equiv ag^x\pmod p; \qquad x\in [K+1, K+M],\,\, y\in [L+1, L+M].
$$
Chan and Shparlinski \cite{SC} used a  sum product estimate of
Bourgain and Garaev \cite{BG} to prove that
$$J_a(M; K, L) < \max \{M^{10/11+o(1)}, M^{9/8+o(1)}p^{-1/8}\}$$
as $M \to \infty$. From our Theorem~\ref{T1} we shall derive the
following improvement on this result.

\begin{corollary}\label{T3} Let $M<t.$ Uniformly over arbitrary integers $K$ and
$L$, we have
$$J_a(M; K, L)<(1+M^{3/4}p^{-1/4})M^{1/2+o(1)}.$$
\end{corollary}
In particular, if $M\le p^{1/3},$ then we have $J_a(M; K, L)<
M^{1/2+o(1)}.$ \

Theorem \ref{T4} allows to strength Corollary \ref{T3} when $ M\ll
p^{3/20}$.
\begin{corollary}\label{T6}
The following bound holds:
$$
J_a(M; K, L)< (1+Mp^{-1/8})M^{1/3+o(1)}.
$$
\end{corollary}
In particular, if $M\ll p^{1/8},$ then we have $J_a(M; K, L)<
M^{1/3+o(1)}.$

\section{Proof of Theorem~\ref{T1}}

We will need the following lemma which is a simple version of a more
precise result about divisors in short intervals, see, for example,
\cite{CJ}.
\begin{lemma}\label{divisors}
For all positive integer $n$  and $m\ge \sqrt n$, the interval
$[m,m+n^{1/6}]$ contains at most two divisors of $n$,
\end{lemma}
\begin{proof}
Suppose that $d_1,d_2,d_3\in [m,m+L]$ are three divisors of $n$. We
claim that the number
$$r=\frac{d_1d_2d_3}{(d_1,d_2)(d_1,d_3)(d_2,d_3)}$$ is also a divisor
of $n$. To see this, for a given prime $q$, let $\alpha_1, \alpha_2,
\alpha_3, \alpha$ such that $q^{\alpha_i}\| d_i,\ i=1,2,3$ and
$q^{\alpha}\| n$. Assume that $\alpha_1\le \alpha_2\le \alpha_3\le
\alpha$. The exponent of $q$ in the rational number $r$ is
$\alpha_1+\alpha_2+\alpha_3-(\min(\alpha_1,\alpha_2)+\min(\alpha_1,\alpha_3)+\min(\alpha_2,\alpha_3))=\alpha_3-\alpha_1.$
Since $0\le \alpha_3-\alpha_1\le \alpha$ we have that $r$ is  an
integer divisor of $n$.

On the other hand, since $(d_i,d_j)\le |d_i-d_j|\le L$ we have
$$n\ge r> \frac{m^3}{L^3}\ge \frac{n^{3/2}}{L^3},$$
and the result follows.
\end{proof}

Now we proceed to prove Theorem~\ref{T1}.  Our approach is based on
Heath-Brown's idea from~\cite{HB}.  We can assume that $M$ is
sufficiently large number. The congruence $xy\equiv \lambda \pmod
p,\ K+1\le x\le K+M,\ L+1\le y\le L+M$ is equivalent to
\begin{equation}\label{1}xy+Kx+Ly\equiv b\pmod p,\quad 1\le x,y\le
M,
\end{equation}
where $b=\lambda-K^2$. From the pigeon-hole principle it follows
that for any positive integer $T<p$  there exists a positive integer
$t\le T^2$ and integers $u_0, v_0$ such that
$$
tK\equiv u_0\pmod p,  \quad tL\equiv v_0\pmod p, \quad |u_0|\le
p/T,\quad |v_0|\le p/T.
$$
From \eqref{1} we get that
$$txy+u_0x+v_0y\equiv b_0 \pmod{p},\quad 1\le x,y\le M,$$
for some  $|b_0|< p/2.$ We write this congruence as an equation
\begin{equation}\label{eq1}txy+u_0x+v_0y= b_0+zp,\quad 1\le x,y\le M,\ z\in \Z.
\end{equation}
Comparing the minimum and maximum value of the left hand side we can
see that
$$|z|\le \Bigl|\frac{txy+u_0x+v_0y-b_0}{p}\Bigr|< \frac{T^2M^2}{p}+\frac{2M}{T}+\frac{1}{2}.$$

We observe  that for each given $z$ the equation \eqref{eq1} is
equivalent to the equation
\begin{equation}\label{eq2}(tx+u_0)(ty+v_0)=n_z,\quad 1\le x,y\le M
\end{equation} for certain integer $n_z$. If $n_z=0,$ then either $tx+u_0=0$ or $ty+v_0=0.$
Since $\lambda\not\equiv 0\pmod p,$ in either case $x$ and $y$
are both determined uniquely. So, we can only consider those $z$ for
which $n_z\not=0.$
\begin{itemize}
\item {Case $M<p^{1/4}/4$}. In this case we take $T=8M$. Then $|z|<1$ and we have to consider
only the integer $n_z=n_0$ in \eqref{eq2}. Each solution of
\eqref{eq2} produces two divisors of $|n_0|$, $|tx+u_0|$ and
$|ty+v_0|$, one of them is greater than or equal to $\sqrt{|n_0|}$.
If $|n_0|\le 2^{36}M^{18}$ the number of solutions of \eqref{eq2} is
bounded by the number of divisors of $n_0$, which is $M^{o(1)}$. If
$|n_0|>2^{36}M^{18}$ the positive integers $|tx+u_0|$ and $|ty+v_0|$
lie in two intervals $\mathcal{I}_1$ and $\mathcal{I}_2$ of length
$T^2M\le 2^6M^3<|n_0|^{1/6}$. If there were five solutions, we would
have three divisors greater of equal to $\sqrt{|n_0|}$ in an
interval of length $\le |n_0|^{1/6}$. We apply Lemma~\ref{divisors}
to conclude that there are at most four solutions. Hence, in this
case we have
$$
I_2(M; K, L)< M^{o(1)}.
$$

\item{Case $M\ge p^{1/4}/4$.} In this case we take
$T\approx (p/M)^{1/3}$. Thus $|z|\ll M^{4/3}/p^{1/3}$. For each $z$
the number of solutions of \eqref{eq2} is bounded by the number of
divisors of $n_z$ which is $p^{o(1)}=M^{o(1)}.$ Hence, in this case
we get
$$
I_2(M; K,L)< \frac{M^{4/3+o(1)}}{p^{1/3}}.
$$
\end{itemize}

Thus, we have proved that
$$
I_2(M; K,L)< \frac{M^{4/3+o(1)}}{p^{1/3}}+M^{o(1)}
$$
which proves the first part of Theorem~\ref{T1}.

The proof of the second part of Theorem~\ref{T1} (corresponding to
the case $K=L$) is similar, with the only difference that  we simply
take $t\le T$ (instead $t\le T^2$) satisfying
$$
tK\equiv u_0\pmod p, \qquad |u_0|\le p/T.
$$

\section{An auxiliary statement}

To prove Theorem \ref{T4} we need the following auxiliary statement.

\begin{proposition}\label{T5}
Let $|A|,|B|,|C|,|D|,|E|,|F|\le M^{O(1)}$ and assume that
$\Delta=B^2-4AC$ is not a perfect square (in particular,
$\Delta\not=0$). Then the diophantine equation
\begin{equation}
\label{eqn:prop}
Ax^2+Bxy+Cy^2+Dx+Ey+F=0
\end{equation}
has at most $M^{o(1)}$ solutions in integers $x,y$ with $1\le
|x|,|y|\le M^{O(1)}.$
\end{proposition}

We shall need several lemmas.
\begin{lemma}
\label{lem:Pell} Let $A$ be a positive integer that is not a perfect
square and let $(x_0,y_0)$ be a solution of the equation the
equation $x^2-Ay^2=1$ in positive integers with the smallest value
of $x_0$. Then for any other integer solution $(x,y)$ there exist a
positive integer
 $n$ such that
$$|x|+\sqrt A|y|=(x_0+\sqrt Ay_0)^n.$$
\end{lemma}

Lemma~\ref{lem:Pell} is well-known from the theory of Pell's
equation.

\begin{lemma}
\label{lem:PellCongruence} Let $A$ be a squarefree integer, $N$ is a
positive integer. Then the congruence $z^2\equiv A\pmod N,\ 0\le
z\le N-1$ has at most $N^{o(1)}$ solutions.
\end{lemma}
\begin{proof}
Let $J(N)$ be the number of solutions of the congruence in question
and let $N=p_1^{\alpha_1}\cdots p_k^{\alpha_k}$ be a canonical
factorization of  $N$. Clearly, $J(N)=J(p_1^{\alpha_1})\cdots
J(p_k^{\alpha_k}),$ where $J(p^{\alpha})$ is the number of solutions
of the congruence $z^2\equiv A\pmod{p^{\alpha}},\ 0\le z\le
p^{\alpha}-1$. Since $A$ is squarefree, we have $J(2^{\alpha})\le 4$
and $J(p^{\alpha})\le 2$ for odd primes $p.$ The result follows.
\end{proof}

\begin{lemma}\label{L}
Let $A,E$ be integers with $|A|,|E| <M^{O(1)}$ such that $A$ is not
a perfect square. Then the equation
\begin{eqnarray*}
x^2-Ay^2=E, \quad 1\le x,y< M^{O(1)}
\end{eqnarray*}
has at most  $M^{o(1)}$ solutions.
\end{lemma}
\begin{proof}
(1) We can assume that $A$ is also a squarefree number. Indeed, let
$A=A_1B_1^2,$ where  $A_1, B_1$ are nonzero integers, $A_1$ is
squarefree and is not a perfect square. Then our equation takes the
form $x^2-A_1(B_1y)^2=E,\ 1\le x,y< M^{O(1)}.$ Since $B_1y<
M^{O(1)},$ it follows that indeed we can assume that $A$ is
squarefree.

\smallskip

(2) We can assume that in our equation $\gcd (x,y)=1$. Indeed, if
$d=\gcd(x,y),$ then $d^2\mid E$. In particular, since $E$ has
$M^{o(1)}$ divisors, we have $M^{o(1)}$ possible values for $d.$
Besides,  $(x/d)^2+A(y/d)^2=E/d^2,$ where we have now
$\gcd(x/d,y/d)=1$. Thus, without loss of generality, we can assume
that $\gcd (x,y)=1$. In particular, it follows that $\gcd(y,E)=1$.

\smallskip

(3) Since $A$ is not a perfect square, we have, in particular, that
$E\ne 0$.

\smallskip

(4) For any $x,y\in \Z_+$ with $(y,E)=1$ there exists $1\le z\le
|E|$ such that $x\equiv zy\pmod{E}$.

\smallskip

Given $1\le z\le |E|,$ let $K_z$ be the set of all pairs $(x,y)$
with
\begin{eqnarray*}
x^2-Ay^2=E, \quad 1\le x,y< M^{O(1)},\quad (x,y)=1
\end{eqnarray*}
such that  $x\equiv zy\pmod{E}$.

If $(x,y)\in K_z,$ then  $(zy)^2-Ay^2\equiv 0\pmod{E}$. Since
$(y,E)=1,$ it follows that  $z^2\equiv A\pmod{E}$. Due to
Lemma~\ref{lem:PellCongruence}, the number of solutions of this
congruence is at most $|E|^{o(1)}=M^{o(1)}$. Thus, we have at most
$M^{o(1)}$ possible values for $z.$ Therefore, it suffices to show
that $|K_z|=M^{o(1)}$ for any such $z$.

Let $x_0$ be the smallest positive integer such that
$$
x_0^2-Ay_0^2=E,\qquad (x_0,y_0)\in K_{z}.
$$
Let $(x,y)$ be any other solution from $K_z.$ Then,
\begin{eqnarray*}
x_0^2-Ay_0^2=E,\quad x^2-Ay^2=E.
\end{eqnarray*}
From this we derive that
\begin{equation}
\label{eq5}
(x_0x-Ayy_0)^2-A(xy_0-x_0y)^2=(x_0^2-Ay_0^2)(x^2-Ay^2)=E^2.
\end{equation}
On the other hand, from $(x_0,y_0),(x,y)\in K_z$ it follows that
\begin{eqnarray*}
x_0\equiv zy_0\pmod{E},\quad x\equiv zy\pmod{E}
\end{eqnarray*}
Since $z^2\equiv A \pmod E,$ we get $xx_0\equiv
z^2yy_0\pmod{E}\equiv Ayy_0\pmod E.$ We also have $x_0y\equiv
xy_0\pmod{E},$ as both hand sides are $zyy_0\pmod{E}$. Therefore,
\begin{equation}\label{eq7}x_0x-Ay_0y\equiv
0\pmod{E},\quad xy_0-x_0y\equiv \pmod{E}.
\end{equation}

From \eqref{eq5} and \eqref{eq7} we get that
$$\left (\frac{x_0x-Ay_0y}{E}\right )^2-A\left (\frac{xy_0-x_0y}{E}\right
)^2=1$$ and the numbers inside of parenthesis are integers.

Now there are two cases to consider:

(1) $A>0$. In view of Lemma~\ref{lem:Pell},
$$\left |\frac{x_0x-Ay_0y}E\right |+\sqrt{|A|}\left
|\frac{xy_0-x_0y}E\right |=(u_0+\sqrt{|A|}v_0)^n,$$ where
$(u_0,v_0)$ is the smallest solution to $X^2-AY^2=1$ in positive
integers, and $n$ is some non-negative integer.

Since the left hand side is of the order of magnitude $M^{O(1)},$ we
have that $n\ll \log M=M^{o(1)}.$ Thus, there are $M^{o(1)}$
possible values for $n$ and, each given $n$ produces at most $4$
pairs $(x,y)$. This proves the statement in the first case.

(2) $A<0$. Then we get that
$$
\frac{x_0x-Ay_0y}E\in \{-1, 0, 1\}, \quad \frac{xy_0-x_0y}{E}\in
\{-1, 0, 1\},
$$
and the result follows.

\end{proof}

\begin{proof}[The proof of Proposition~\ref{T5}]
Now we can deduce Proposition~\ref{T5} from Lemma~\ref{L}.
Multiplying~\eqref{eqn:prop} by $4A,$ we get
\begin{equation*}
(2Ax+By+D)^2-\Delta y^2+(4EA-2BD)y+4AF-D^2=0,
\end{equation*}
where $\Delta=B^2-4AC$. Multiplying by $\Delta$ we get,
\begin{equation*}
(\Delta y+BD-2EA)^2-\Delta(2x+By+D)^2=T,
\end{equation*}
where $T=(BD-2EA)^2+\Delta(4AF-D^2).$ Now, since $\Delta$ is not a
full square, and since $T, \Delta\le M^{O(1)},$ we have, by
Lemma~\ref{L} and the condition $|A|, |B|, |C|, |D|,|E|, |F|\le M$,
that there are at most $M^{o(1)}$ possible pairs $(\Delta y+BD-2EA,
2x+By+D).$ Each such pair uniquely determines $y$ (since
$\Delta\not=0$) and $x.$ This finishes the proof of
Proposition~\ref{T5}.
\end{proof}
\section{Proof of Theorem \ref{T4}}

In what follows, by $v^*$ we denote the least positive integer such
that $vv^*\equiv 1\pmod p.$ We rewrite our congruence in the form
\begin{eqnarray}\label{equiv0} (L+x)(L+y)(L+z)\equiv \lambda \pmod
p,\quad 1\le x,y,z\le M \nonumber
\end{eqnarray}
which, in turn,  is equivalent to the congruence
\begin{eqnarray}\label{equiv}
L^2(x+y+z)+L(xy+xz+yz)+xyz\equiv \lambda-L^3\pmod p,\quad 1\le
x,y,z\le M.
\end{eqnarray}

Assume that $M\ll p^{1/8}$ and that $p$ is large enough to satisfy
several inequalities through the proof. Let
\begin{equation}\label{k}
k=\max\{1, 2M^2/p^{1/4}\}.
\end{equation}
\begin{lemma}\label{auxiliar}
If $L=uv^*$ for some integers $u,v$ with $|u|\le M^3/k$ and $1\le
|v|\le M^2/k,$ then the number of solutions of the
congruence~\eqref{equiv0} is at most $M^{o(1)}$.
\end{lemma}
\begin{proof}
The congruence~\eqref{equiv} is equivalente to
$$v^2xyz+uv(xy+xz+yz)+u^2(x+y+z)\equiv \mu \pmod p,$$
where $|\mu|<p/2$ and $\mu\equiv \lambda v^2-u^3v^*$. The absolute
value of the left hand side is bounded by
\begin{eqnarray*}(M^2/k)^2M^3+(M^3/k)(M^2/k)(3M^2)+(M^3/k)^2(3M)&\le &7M^7/k^2\le
7M^7/(2M^2/p^{1/4})^2\\&=&\frac{7}4M^3p^{1/2} <p/2.\end{eqnarray*}
 Hence, the congruence~\eqref{equiv} is equivalent to the
equality
$$
v^2xyz+uv(xy+xz+yz)+u^2(x+y+z)= \mu.
$$
Multiplying by $v,$ we get
$$(vx+u)(vy+u)(vz+u)=v\mu +u^3$$
The absolute value of the right and the left hand sides is $\le
M^{O(1)}$, and besides it is distinct from zero (since
$v\mu+u^3\equiv \lambda v^3\pmod p,$ and $\lambda v^3\not\equiv
0\pmod p.$ Therefore, the number of solutions of the latter equation
is bounded by $M^{o(1)}$ and the lemma follows.
\end{proof}

Due to this lemma, from now on we can assume that $L$ does not
satisfy the condition of Lemma~\ref{auxiliar}, that is
\begin{equation}\label{suposicion}
L\ne uv^*,\qquad |u|\le M^3/k,\quad |v|\le M^2/k.
\end{equation}

\

For $0\le r,s\le 3k-1$ and $0\le t\le k-1$ let $S_{r,s,t}$ be the
set of solutions $(x,y,z)$ such that
$$\begin{cases}\ x+y+z\in (\frac{rM}k,\frac{(r+1)M}k]\\
xy+xz+yz\in (\frac{sM^2}k,\frac{(s+1)M^2}k]\\xyz\in
(\frac{tM^3}k,\frac{(t+1)M^3}k]\end{cases}$$ Clearly, the number of
solutions $I_3(M;L)$ of our congruence satisfies
$$
I_3(M;L)\le 9k^3\max|S_{rst}|.
$$

We fix one solution $(x_0,y_0,z_0)\in S_{rst}.$ Any other solution
$(x_i,y_i,z_i)\in S_{rst}$ satisfies the congruence
\begin{equation}\label{eq4}
A_iL^2+B_iL+C_i\equiv 0\pmod p
\end{equation}
where
\begin{eqnarray*}&A_i&=x_i+y_i+z_i-(x_0+y_0+z_0),\\ &B_i&=x_iy_i+x_iz_i+y_iz_i-(x_0y_0+x_0z_0+y_0z_0 ),\\
&C_i&=x_iy_iz_i-x_0y_0z_0.\end{eqnarray*} We have
\begin{equation}\label{cotas}|A_i|\le M/k,\ |B_i|\le M^2/k,\ |C_i|\le
M^3/k.\end{equation}

A solution $(x_i,y_i,z_i)\ne (x_0,y_0,z_0)$ we call degenerated if
$A_i= 0,$ and non-degenerated otherwise.

\

\textbf{The set of non-degenerated solutions.}

\smallskip

We shall show that there are at most $M^{o(1)}$  non-degenerated
solutions. So that, let us assume that there are at least several
non-degenerated solutions. With this set of solutions we shall form
a system of congruence with respect to $L,L^2$. Let us fix one
solution $(A_1,B_1,C_1).$ Note that the condition $A_i\not=0$
implies that $A_i\not\equiv 0\pmod p.$
\smallskip

Case (1).  If $A_iB_1\ne A_1B_i$ for some $i$, then in view of
inequalities~\eqref{cotas} we also have that $A_iB_1\not\equiv
A_1B_i\pmod p.$ Solving the system of equations~\eqref{eq4}
corresponding to the indices $i$ and $1,$ we obtain that
$$L\equiv (C_iA_1-A_iC_1)(A_iB_1-A_1B_i)^*\pmod p\equiv uv^*\pmod p,$$
$$L^2\equiv (B_iC_1-C_iB_1)(A_iB_1-A_1B_i)^*\pmod p\equiv u'v^*\pmod p,$$
where
$$
u=C_iA_1-A_iC_1, \quad v=A_iB_1-A_1B_i, \quad u'=B_iC_1-C_iB_1.
$$
From this we derive that
\begin{equation}\label{uv}|u|\le 2M^4/k^2,\ |u'|\le 2M^5/k^2,\ |v|\le
2M^3/k^2\end{equation} and $(uv^*)^2\equiv L^2\pmod p\equiv
u'v^*\pmod p.$ Hence, $u^2\equiv u'v\pmod p$ and,
using~\eqref{uv},~\eqref{k}, we get $|u^2|,|u'v|\le 4M^8/k^4\le
p/4,$ so that we actually have the equality $u^2=u'v$.

\

Multiplying~\eqref{equiv} by $v,$ we get
\begin{equation}
\label{eqn:vxyz+u(xy+yz+zx)}
 vxyz+u(xy+xz+yz)+u'(x+y+z)\equiv
v(\lambda -L^3)\pmod p
\end{equation}
Since $1 \le x,y,z\le M$, the inequalities \eqref{uv} give
\begin{equation*}
|vxyz+u(xy+xz+yz)+u'(x+y+z)|\le \frac{14M^6}{k^2}\le
\frac{14M^6}{(2M^2p^{-1/4})^2}=\frac{7M^2p^{1/2}}2< p/2.
\end{equation*}
This converts the congruence~\eqref{eqn:vxyz+u(xy+yz+zx)} into the
equality
\begin{equation*}
vxyz+u(xy+xz+yz)+u'(x+y+z)=\mu
\end{equation*}
for some $\mu\ll M^{O(1)}$ and $\mu\equiv v(\lambda -L^3)\pmod p.$
We multiply this equality by $v^2$ and use $u'v=u^2;$ we get
that\begin{equation} (vx+u)(vy+u)(vz+u)=\mu v^2+u^3.
\end{equation}
Since $\mu v^2+u^3\ne 0,$ the total number of solutions of the
latter equation is  $\ll M^{o(1)}$.

\smallskip

Case (2). If we are not in case (1), then for any index $i$ one has
$A_1B_i=A_iB_1,$ which, in turn, implies that we also have
$$
A_1C_i\equiv A_iC_1\pmod p.
$$
In view of inequalities~\eqref{cotas}, we get that the latter
congruence is also an equality, so that we have
\begin{eqnarray}
A_1B_i= A_iB_1,\quad A_1C_i= A_iC_1.
\end{eqnarray}
From the first equation and the definition of $A_i,B_i,C_i,$ we get
\begin{equation}
\label{eqn:FirstFinal}
z_i(A_1(x_i+y_i)-B_1)=B_1(x_i+y_i-a_0)-A_1x_iy_i+b_0A_1,
\end{equation}
from the second equation we get
\begin{equation}
\label{eqn:SecondFinal} z_i(A_1x_iy_i-C_1)= C_1(x_i+y_i-a_0)+c_0A_1,
\end{equation}
where
$$
a_0=x_0+y_0+z_0,\quad b_0=x_0y_0+y_0z_0+z_0x_0,\quad c_0=x_0y_0z_0.
$$
Multiplying~\eqref{eqn:FirstFinal} by $A_1x_iy_i-C_1$,
and~\eqref{eqn:SecondFinal} by $A_1(x_i+y_i)-B_1$, subtracting the
resulting equalities, and making the change of variables
$x_i+y_i=u_i,\ x_iy_i=v_i,$ we obtain
\begin{equation*}
\left (B_1(u_i-a_0)-A_1v_i+b_0A_1\right )\left (A_1v_i-C_1\right
)=\left (C_1(u_i-a_0)+c_0A_1\right )\left (A_1u_i-B_1 \right ).
\end{equation*}
We rewrite this equation in the form
 \begin{equation*}
A_1v_i^2+C_1u_i^2-B_1u_iv_i-(a_0C_1-c_0A_1)u_i-(b_0A_1-a_0B_1+C_1)v_i+b_0C_1-c_0B_1
=0.
\end{equation*}

If $B_1^2-4A_1C_1$ is a full square (as a number), say $R_1^2$, then
from~\eqref{eq4} we obtain that $L\equiv(-B_1\pm R_1)(2A_1)^*=uv^*$
with $|u|\le |B_1|+|B_1|+\sqrt{|4A_1C_1|}\le 4M^2/k,\ |v|\le 2M/k$,
which contradicts our condition~\eqref{suposicion}.

If $B_1^2-4A_1C_1$ is not a full square, then we are at the
conditions of Proposition~\ref{T5} and we can claim that the number
of pairs $(u_i,v_i)$ is at most $M^{o(1)}$. We now conclude the
proof observing that each pair $u_i,v_i$ produces at most two pairs
$x_i,y_i$, which, in turn, determines $z_i$. Therefore, the number
of non-degenerated solutions counted in $S_{rst}$ is at most
$M^{o(1)}.$

\

\textbf{The set of degenerated solutions.}

\smallskip

We now consider the set of solutions for which $A_i=0.$ If $B_i\ne
0,$ then $B_i\not\equiv 0 \pmod p$ and thus we get  $L=-C_iB_i^*$
with $|C_i|\le M^3/k,\ |B_i|\le M^2/k$, which contradicts
condition~\eqref{suposicion}.

If $B_i=0$ then together with $A_i=0$ this implies that $C_i=0.$
Thus,
\begin{eqnarray*}
x_i+y_i+z_i= a_0=x_0+y_0+z_0,\\
x_iy_i+x_iz_i+y_iz_i= b_0=x_0y_0+y_0z_0+z_0x_0,\\
x_iy_iz_i= c_0=x_0y_0z_0.
\end{eqnarray*}
Hence,
$$
(L+x_i)(L+y_i)(L+z_i)= (L+x_0)(L+y_0)(L+z_0).
$$
The right hand side is not zero (since it is congruent to
$\lambda\pmod p$ and $\gcd(\lambda,p)=1$). Thus, the number of
solutions of this equation is at most $M^{o(1)}.$ The result
follows.

\section{Proof of Corollaries}

If $M<p^{5/8}$ then
$$
\frac{M^{4/3+o(1)}}{p^{1/3}}+M^{o(1)}<M^{4/5+o(1)}
$$
and the statement of Corollary~\ref{T2} for $I_2(M; K, L)$ follows
from Theorem~\ref{T1}. If $M>p^{5/8}$ then, $p^{1/2}(\log
p)^2<M^{4/5+o(1)}$ and the statement of Corollary \ref{T2} for
$I_2(M; K, L)$ follows from~\eqref{1}. Analogously we deal with
$I_2(M; K, K)$ considering the cases $M>p^{2/3}$ and $M<p^{2/3}.$

\smallskip

In order to prove Corollary~\ref{T3}, let $k=J_a(M; K, L)$ and let
$(x_i,y_i),\, i=1,\ldots, k,$ be all solutions of the congruence
 $y\equiv ag^x\pmod p$ with $x_i\in [K+1, K+M]$ and $y_i\in [L+1, L+M].$ Since $M<t,$
 the numbers
 $y_1,\ldots, y_k$ are distinct. Since $y_iy_j\equiv
ag^z\pmod p$ for some $z\in [2K+2, 2K+2M]$, there exists a value
$\lambda$ such that for at least $k^2/2M$
 pairs $(y_i,y_j)$ we have $y_iy_j\equiv
\lambda \pmod p$. Hence, theorem~\ref{T1} implies that
$$
\frac{k^2}{2M}<\frac{M^{3/2+o(1)}}{p^{1/2}}+M^{o(1)},
$$
and the result follows.

Corollary~\ref{T6} is proved similar to Corollary~\ref{T3}. For any
triple $(i,j,\ell)$ we have $y_iy_jy_{\ell}\equiv ag^z\pmod p$ for
some $z\in [3K+3, 3K+3M].$ Hence, there exists $\lambda\not\equiv
0\pmod p$ such that the congruence $y_iy_jy_{\ell}\equiv
\lambda\pmod p$ has at least $k^{3}/3M$ solutions. Thus,
$$
\frac{k^3}{3M}<M^{o(1)},
$$
and the result follows in this case. If  $M>p^{1/8},$ then in the
interval $[L+1, L+M]$ we can find a subinterval of length $p^{1/8}$
which would contain at least $k/(2Mp^{-1/8})$ members from
$y_1,\ldots,y_k.$ Thus, the preceding argument gives that
$$
\frac{\Bigl(\frac{k}{Mp^{-1/8}}\Bigr)^3}{3M}<M^{o(1)},
$$
and the result follows.

 Now we prove Corollary~\ref{T7}. Let $W$ be
the number of solutions of the congruence
$$
xyz\equiv x'y'z'\pmod p,\quad (x,x',y,y',z,z')\in
\mathcal{I}_1\times\mathcal{I}_1\times\mathcal{I}_2\times\mathcal{I}_2\times\mathcal{I}_3\times\mathcal{I}_3.
$$
Then,
$$
W=\frac{1}{p}\sum_{\chi}\Bigl|\sum_{x\in
\mathcal{I}_1}\chi(x)\Bigr|^2\Bigl|\sum_{y\in
\mathcal{I}_1}\chi(y)\Bigr|^2\Bigl|\sum_{z\in
\mathcal{I}_1}\chi(z)\Bigr|^2.
$$
Applying the Holder's inequality, we obtain
$$
W\le\Bigl(\frac{1}{p}\sum_{\chi}\Bigl|\sum_{x\in
\mathcal{I}_1}\chi(x)\Bigr|^6\Bigr)^{1/3}\Bigl(\frac{1}{p}\sum_{\chi}\Bigl|\sum_{y\in
\mathcal{I}_2}\chi(y)\Bigr|^6\Bigr)^{1/3}\Bigl(\frac{1}{p}\sum_{\chi}\Bigl|\sum_{z\in
\mathcal{I}_3}\chi(z)\Bigr|^6\Bigr)^{1/3}.
$$
Thus,
$$
W\le W_1^{1/3}\cdot W_2^{1/3}\cdot W_3^{1/3},
$$
where $W_j$ is the number of solutions of the congruence
$$
xyz\equiv x'y'z'\pmod p,\quad x,y,z,x',y',z'\in \mathcal{I}_j.
$$
According to Theorem~\ref{T4}, for each given triple $(x',y',z')$
there are at most $|\mathcal{I}_j|^{o(1)}$ possibilities for
$(x,y,z).$ Thus, we have that $W_i\le |\mathcal{I}_j|^{3+o(1)}.$
Therefore,
$$
W\le
(|\mathcal{I}_1|\cdot|\mathcal{I}_2|\cdot|\mathcal{I}_3|)^{1+o(1)}.
$$
Now, using the well known relationship between the cardinality of a
product set and the number of solutions of the corresponding
equation, we get
$$
|\mathcal{I}_1\cdot\mathcal{I}_2\cdot\mathcal{I}_3|\ge
\frac{|\mathcal{I}_1|^2\cdot|\mathcal{I}_2|^2\cdot|\mathcal{I}_3|^2}{W}\ge
(|\mathcal{I}_1|\cdot|\mathcal{I}_2|\cdot|\mathcal{I}_3|)^{1-o(1)}
$$
and the result follows.

\section{Conjectures and Open problems}

We conclude our paper with several conjectures and open problems.

\begin{conjecture}
For $M<p^{1/2}$ one has $I_2(M; K, L)<M^{o(1)}$
\end{conjecture}

\begin{conjecture}
For $M<p^{1/3}$ one has $I_3(M; L)<M^{o(1)}$
\end{conjecture}

\begin{conjecture}
For $M<p^{1/2}$ one has $J_a(M; K, L)< M^{o(1)}.$
\end{conjecture}

\begin{conjecture}
Let $\mathcal{I}_1, \mathcal{I}_2, \mathcal{I}_3$ be intervals in
$\F^*_p$ of length $|\mathcal{I}_i|< p^{1/3}.$ Then
$$
|\mathcal{I}_1\cdot \mathcal{I}_2\cdot \mathcal{I}_3|=
(|\mathcal{I}_1|\cdot |\mathcal{I}_2|\cdot
|\mathcal{I}_3|)^{1-o(1)}.
$$
\end{conjecture}

\begin{problem}
From Theorem~\ref{T1} it follows that if if $M<p^{1/4},$ then
$I_2(M; K, L)< M^{o(1)}.$ Improve the exponent $1/4$ to a larger
constant.
\end{problem}

\begin{problem}
From Theorem~\ref{T1} it follows that if $M<p^{1/3},$ then $I_2(M;
L, L)< M^{o(1)}.$ Improve the exponent $1/3$ to a larger constant.
\end{problem}

\begin{problem}
Theorem~\ref{T4} claims that if $M<p^{1/8},$ then $I_3(M; L)<
M^{o(1)}$. Improve the exponent $1/8$ to a larger constant.
\end{problem}

\end{document}